\documentclass[a4paper]{article}
\usepackage{amsthm,amsfonts,amsmath,amssymb}
\usepackage[cp1251]{inputenc}
\usepackage[english]{babel}
\usepackage[final]{graphicx}
\usepackage{setspace}
\usepackage[12pt]{extsizes}
\begin{document}
\newcommand{\per}{{\rm per}}
\newtheorem{teorema}{Theorem}
\newtheorem{lemma}{Lemma}
\newtheorem{utv}{Proposition}
\newtheorem{svoistvo}{Property}
\newtheorem{sled}{Corollary}
\newtheorem{con}{Conjecture}

\author{A. A. Taranenko}
\title{On the numbers of 1-factors and 1-factorizations of hypergraphs}
\date{}

\maketitle

\begin{abstract}
A 1-factor of a hypergraph $G=(X,W)$ is a set of hyperedges such that every vertex of $G$ is incident to exactly one hyperedge from the set. A 1-factorization is a partition of all hyperedges of $G$ into disjoint 1-factors. 
The adjacency matrix of a $d$-uniform hypergraph $G$ is the $d$-dimensional (0,1)-matrix of order $|X|$ such that an element $a_{\alpha_1, \ldots, \alpha_d}$ of $A$ equals 1 if and only if $\left\{\alpha_1, \ldots, \alpha_d\right\}$ is a hyperedge of $G$.
Here we estimate the number of 1-factors of uniform hypergraphs and the number of 1-factorizations of complete uniform hypergraphs by means of permanents of their adjacency matrices.
\end{abstract}

\section{Introduction}

Let $G=(V,E)$ be a graph on $n$ vertices. The adjacency matrix $M(G)$ of $G$ is the (0,1)-matrix of order $n$ such that the entry $m_{i,j}$ equals one if and only if the vertices $i$ and $j$ are adjacent. A 1-factor (perfect matching) of the graph $G$ is a 1-regular subgraph that has the same vertex set as $G$.  A 1-factorization of $G$ is a partition of the edges of the graph into disjoint 1-factors. 

It is well known that the number of 1-factors of a bipartite graph with equal parts  coincides with the permanent of its biadjacency matrix (rows of this matrix correspond to the first part, and the columns, to the second). At the same time, this number is equal to the square root of the permanent of the adjacency matrix.  In~\cite{alon}, Alon and Friedland proved that the number of 1-factors of any graph is not greater than  the square root of the permanent of the adjacency matrix. Also, this result was stated in~\cite{gibson}.

Permanents can be used for the estimation of the number of 1-factorizations $\Phi(n)$ of the complete graph $K_n$~\cite{isr}:
$$ \left((1+o(1)) \frac{n}{4e^2}\right) ^{\frac{n^2}{2}} \leq \Phi(n) \leq \left((1+o(1)) \frac{n}{e^2}\right) ^{\frac{n^2}{2}}. $$
The lower bound was obtained by Cameron in~\cite{cam}. This proof requires the van der Waerden conjecture which was proved by Egorychev in~\cite{egor} and Falikman in~\cite{falikman}. 
The upper bound follows from Bregman's theorem for the permanent of (0,1)-matrices~\cite{bregman} and from the result of Alon and Friedland~\cite{alon}.

Of course, there exist several bounds on the number of 1-factorizations of other graphs. For example, for $d$-regular  bipartite  graphs we have the following result proved by Schrijver in~\cite{shriv}:

\begin{teorema} \label{tchriv}
Let $G$ be a $d$-regular bipartite graph on $2n$ vertices. Then the number of 1-factorizations of $G$ is not less than $\left( \frac{d!^2}{d^d}\right)^n.$
\end{teorema}

The main aim of this paper is to estimate the number of 1-factors and 1-factorizations of uniform hypergraphs by means of permanents of multidimensional matrices. For this purpose we need the following definitions.

Let $n,d \in \mathbb N$, and let $I_n^d= \left\{ (\alpha_1, \ldots , \alpha_d):\alpha_i \in \left\{1,\ldots,n \right\}\right\}$.
A \textit{$d$-dimensional matrix $A$ of order $n$} is an array $(a_\alpha)_{\alpha \in I^d_n}$, $a_\alpha \in\mathbb R$.

Let $k\in \left\{0,\ldots,d\right\}$. A \textit{$k$-dimensional plane} in $A$ is the submatrix of $A$ obtained by fixing $d-k$ indices and letting the other $k$ indices vary from 1 to $n$. The \textit{direction} of a plane is the (0,1)-vector describing which indices are fixed in the plane. A $(d-1)$-dimensional plane is said to be a \textit{hyperplane}, and let us enumerate all hyperplanes of one direction by numbers $1, \ldots, n.$

For a $d$-dimensional matrix $A$ of order $n$, denote by $D(A)$ the set of its diagonals
$$D(A)=\left\{ (\alpha^1,\ldots,\alpha^n) | \alpha^i \in I_n^d, \forall i\neq j~ \rho (\alpha^i,\alpha^j)=d\right\},$$
where $\rho$ is the Hamming distance (the number of positions at which the corresponding entries are different) and where a diagonal $(\alpha^1, \ldots, \alpha^n)$ should be considered as an unordered set.
Then the \textit{permanent} of a matrix $A$ is 
$$\per A = \sum\limits_{p\in D} \prod\limits_{\alpha \in p} a_\alpha.$$
In other words, the permanent of a $d$-dimensional matrix $A$ of order $n$ is
$$\per A = \sum\limits_{\sigma_1, \ldots, \sigma_{d-1} \in S_n} \prod\limits_{i=1}^n a_{i, \sigma_1(i), \ldots, \sigma_{d-1}(i)},$$
where $\sigma_1, \ldots, \sigma_{d-1}$ are permutations from the symmetric group $S_n$.

Permanents are often useful in the estimation of the number of some combinatorial structures. For instance, Theorem~\ref{tchriv} is a simple corollary from the following result of~\cite{shriv}:

\begin{teorema} \label{nper}
Let $A$ be a matrix of order $n$ with nonnegative integer entries whose row and column sums are equal to $k$. Then
$$\per A \geq \left(\frac{(k-1)^{k-1}}{k^{k-2}}\right)^n.$$
\end{teorema}

A strengthening of this result for non-negative real matrices where each row contains not more than $k$ non-zero entries is obtained in~\cite{gurvits}.
For additional information about permanents of 2-dimensional matrices we refer the reader to the classic book~\cite{minc}.

There are very few bounds on the permanent of multidimensional matrices. The simplest one can be proved by induction on the order of matrices:

\begin{utv} \label{triv}
Let $A$ be a $d$-dimensional (0,1)-matrix of order $n$.  Suppose that for $i \in \left\{1, \ldots, n\right\}$ the number of ones in the $i$th hyperplane of some direction of the matrix $A$ is not greater than $r_i$. Then
$$\per A \leq \prod \limits_{i=1} ^n r_i.$$
\end{utv}

The first nontrivial upper bound on the permanent of 3-dimensional (0,1)-matrices was proved by Dow and Gibson~\cite{dow}:
\begin{teorema} \label{tdow}
Let $A$ be a 3-dimensional (0,1)-matrix of order $n$. Suppose that for $i \in \left\{1, \ldots, n\right\}$ the number of ones in the $i$th hyperplane of some direction of the matrix $A$ is not greater than $r_i$. Then
$${\rm per} A \leq \prod\limits_{i=1}^n r_i!^{1/r_i}.$$
\end{teorema}

Several attempts to obtain exact bounds on multidimensional permanents were made in~\cite{my}.

Next recall some definitions on hypergraphs.

The pair $G=(X,W)$ is called a \textit{$d$-uniform hypergraph} on $n$ vertices  with vertex set $X$  and  hyperedge set $W$ if $|X|=n$ and each hyperedge $w \in W$ is a set of $d$ vertices.  A hypergraph $G$ is called \textit{simple} if it has no multiple hyperedges.
The  \textit{degree} of a vertex $x \in X$ in a hypergraph $G$ is the number of hyperedges containing $x$. 

A hypergraph $G=(X,W)$ is  \textit{connected} if for all vertices $x_1, x_2 \in X$ there exists a sequence  of hyperedges $w_1, \ldots, w_l$ such that $x_1 \in w_1$, $x_2 \in w_l$, and for all $i \in \left\{1, \ldots, l-1\right\}$ it holds $w_i \cap w_{i+1} \neq \emptyset.$

A \textit{1-factor} of a hypergraph $G$ is a set of hyperedges  such that every vertex of the hypergraph is incident to exactly one hyperedge from the set. A \textit{1-factorization} of $G$ is an ordered partition of all hyperedges of the hypergraph into disjoint 1-factors. Denote by $\varphi(G)$ the number of 1-factors of $G$, and by $\Phi(G)$ the number of 1-factorizations of $G$. 

A hypergraph $G$ is said to be \textit{1-factorable} if it admits a 1-factorization.
A $d$-uniform hypergraph  $G=(X,W)$ in which all vertices have the same degree $k$ is called a \textit{$d$-uniform $k$-factor}.

An \textit{$n$-balanced $d$-partite} hypergraph is a $d$-uniform hypergraph such that its vertex set consists of $d$ parts of size $n$, each hyperedge containing precisely one vertex from each part.

The \textit{incidence matrix} of a hypergraph $G$ is a $|X| \times |W|$ matrix $(b_{i,j})$ such that $b_{i,j} = 1$ if the vertex $x_i$ and the hyperedge $w_j$ are incident, and 0 otherwise.  
The \textit{adjacency matrix} $M(G)$ of a $d$-uniform hypergraph $G$ is the $d$-dimensional (0,1)-matrix of order $n$ such that an entry $m_{\alpha_1, ... ,\alpha_d}$ equals one if and only if $\{x_{\alpha_1}, ... ,x_{\alpha_d}\}$ is a hyperedge of $G$.

The paper is organized as follows. In the next section we prove that the number of 1-factors of a simple $d$-uniform hypergraph $G$ on $n$ vertices satisfies
$$\varphi(G) \leq \left(\frac{\per M(G)}{\mu(n,d)} \right)^{1/d},$$
where $\mu(n,2)=1,$ $\mu(n,3)= \left( \frac{2^{3/2}}{3} \right)^n$ for all integer $n$ and
$$\mu(n,d) = \left( \frac{d!^{2}}{d^{d} d!^{1/d}} \right) ^n$$
for all $d \geq 4.$

Then using this result, we prove an asymptotic upper bound on the number of 1-factorizations of the $d$-uniform complete hypergraph $G^d_n$ on $n$ vertices:
$$\Phi(G^d_n) \leq \left((1+o(1))\frac{ n^{d-1}}{\mu(n,d)^{1/n} e^d} \right)^{\frac{n^d}{d!}} \mbox{ as }n \rightarrow \infty.$$

\section{An upper bound on the number of 1-factors of hypergraphs}

Let $G$ be a simple $d$-uniform hypergraph on $n$ vertices.  It is clear that if there exists a 1-factor of the hypergraph $G$, then the number of vertices $n$ is divisible by $d$. Therefore, below we will consider only $n$ which is a multiple of $d$.

As for graphs, the characterization problem of hypergraphs having a 1-factor is rather complicated. Often sufficient conditions on such hypergraphs can be expressed as demanding that each $k$-element set of vertices is covered by a sufficiently large number of hyperedges. There are many papers concerning the existence problem of 1-factors of hypergraphs (see, for example,  \cite{lo,rodl,zhao}). 

It is easy to prove that the number of 1-factors of a hypergraph $G$ is not greater than the permanent of its adjacency matrix. Indeed, let the hyperedges $e_1, \ldots, e_{n/d}$ form a 1-factor in $G$. Fix some permutation of vertices for each $e_i$ and construct ordered $d$-tuples $\alpha^1, \ldots, \alpha^n$ by a cyclic permutation. By the definition of adjacency matrix, we have $m_{\alpha^i}=1$ for all $i=1,\ldots,n.$ Moreover, the Hamming distance between different  $\alpha^i$ and $\alpha^j$ is equal to $d$. Consequently, $\alpha^1, \ldots, \alpha^n$ form a unity diagonal in the matrix $M(G)$, and $\varphi(G) \leq \per M(G).$

The main result of this section is the following theorem, that strengthens this bound:

\begin{teorema} \label{fact}
Let $G$ be a simple $d$-uniform hypergraph on $n$ vertices, and let $d$ divide $n$. Define the function $\mu(n,d)$ by $\mu(n,2) =1,$ $\mu(n,3)= \left( \frac{2^{3/2}}{3} \right)^n$  for all integer $n$ and
$$\mu(n,d) = \left( \frac{d!^{2}}{d^{d} d!^{1/d}} \right) ^n$$
for all $d \geq 4.$
Then the number of 1-factors of the hypergraph $G$ satisfies
$$\varphi(G) \leq \left(\frac{\per M(G)}{\mu(n,d)} \right)^{1/d}.$$
\end{teorema}

Let us first prove the easy corollaries of the theorem:

\begin{sled} \label{factsled}
If $d \neq 3$, then the number of 1-factors of a simple $d$-uniform hypergraph is not greater than the $d$th root of the permanent of its adjacency matrix:
$$\varphi(G) \leq (\per M(G))^{1/d}.$$
\end{sled}

\begin{proof}
It is sufficient to note that the function  $\mu(n,d) $ is not less than one for all $d\geq 4$ and $n\geq d$.
\end{proof} 

Note that the case $d =3$ is exceptional. Despite our efforts, we are not succeed in a proof of the following statement that is likely to be true.

\begin{con}
The number of 1-factors of a simple $3$-uniform hypergraph is not greater than the cube root of the permanent of its adjacency matrix.
\end{con}

Theorem~\ref{fact} also allows us to bound the number of 1-factors of a hypergraph in terms of the degrees of its vertices.

\begin{sled}
Let $G=(X,W)$ be a simple $d$-uniform hypergraph on $n$ vertices, and let the vertex $x_i \in X$ have degree $r_i$. Then the number of 1-factors of the hypergraph $G$ satisfies
$$\varphi(G) \leq \left(\frac{(d-1)!^n}{\mu(n,d)} \prod\limits_{i=1}^n r_i \right)^{1/d}.$$
\end{sled}

\begin{proof}
Consider the $i$th hyperplane in $k$th direction of the adjacency matrix $M(G)$. By definition, the entries in that hyperplane are parameterized by $d$-tuples of vertices, where the $k$th position in the tuple is occupied by the $i$th vertex. Hence every hyperedge containing the $i$th vertex results in $(d-1)!$ ones in the $i$th hyperplane that correspond to $(d-1)!$ permutations of the remaining vertices of the hyperedge. Therefore the number of ones in $i$th hyperplane of some direction equals  $r_i(d-1)!$. Using Proposition~\ref{triv}, we obtain $\per M(G) \leq (d-1)!^n \prod\limits_{i=1}^n r_i$.
\end{proof}

Let us begin the proof of Theorem~\ref{fact}. For this purpose we need some auxiliary constructions.

Denote by  $\mathfrak{F}(G)$ the set of all ordered $d$-tuples of 1-factors of $G$, where $d$-tuples can contain identical 1-factors. It is clear that $|\mathfrak{F}(G)| = \varphi^d(G)$.

Let $f \in \mathfrak{F}(G)$ be an ordered $d$-tuple of 1-factors. Consider the $d$-uniform hypergraph $F$ on $n$ vertices such that its hyperedge set is exactly the set of all hyperedges of the $d$-tuple $f$, and the multiplicities of hyperedges from $F$ and $f$ are the same. By construction, $F$ is a 1-factorable $d$-uniform $d$-factor. Denote by $\Phi(F)$ the number of all 1-factorizations of $F$ (i.e., the number of all $d$-tuples $f \in \mathfrak{F}(G)$ corresponding to $F$).

Let $w$ be a hyperedge of a hypergraph $G$. An arbitrary ordering of vertices of a hyperedge $w$ is said to be an \textit{orientation} of $w$.  

An \textit{orientation} of a hypergraph $G$ is the set of orientations of all its hyperedges, where each copy of multiple hyperedges is oriented separately. A \textit{proper orientation} of a hypergraph $G$ is an orientation such that there are no vertices having the same position in different orientations of hyperedges. Let $\delta(G)$ be the set of all proper orientations of $G$, and let $\Delta(G) = |\delta(G)|$ be the cardinality of this set.

It is not hard to prove that in case $d=2$ the set of proper orientations $\delta(F)$ of a 1-factorable 2-uniform 2-factor $F$ is not empty. Indeed, in this case the graph $F$ is a union of even cycles. Choose a tour in each cycle, and orient edges according to the tours. It can be checked that such orientation is proper. Later, we show that each  1-factorable $d$-uniform $d$-factor $F$ has a proper orientation.

Let $F_1$ and $F_2$ be 1-factorable $d$-uniform $d$-factors. Note that if the hyperedge sets of $F_1$ and $F_2$ are the same (taking into account the multiplicity of hyperedges), then $\delta(F_1)=\delta(F_2)$. If the hyperedge sets of $F_1$ and $F_2$ are different, then all orientations of $F_1$ and $F_2$ will be different too, and  $\delta(F_1) \cap \delta(F_2) = \emptyset.$

Therefore all $d$-tuples from  $\mathfrak{F}(G)$  can be divided into classes such that $d$-tuples from one class induce the same $d$-uniform $d$-factor $F$, the cardinality of each class equals $\Phi(F)$, and the sets of proper orientations for different classes are disjoint.

Now we give the key statement for the proof of Theorem~\ref{fact}: 

\begin{utv} \label{base}
Let $F$ be a 1-factorable $d$-uniform $d$-factor. Then
 $$\Phi(F) \leq \frac{\Delta(F)}{\mu(n,d)}.$$
\end{utv}

Using this proposition, it is quite easy to prove Theorem~\ref{fact}.

Let $G$ be a simple $d$-uniform hypergraph. Put $\gamma(G) = \bigcup \delta(F)$, where the union is over all $d$-uniform $d$-factors $F$ constructed by all $f \in \mathfrak{F}(G)$. Note that the set of entries of the adjacency matrix $M(G)$, whose indices make a proper orientation from $\gamma(G)$, forms a unity diagonal in $M(G)$. Consequently,  $|\gamma(G)| \leq \per M(G).$ The following is a simple corollary to Proposition~\ref{base}:

\begin{sled} \label{factdiag}
Let $G$ be a simple $d$-uniform hypergraph on $n$ vertices. Then
$$|\mathfrak{F}(G)| \leq \frac{|\gamma(G)|}{\mu(n,d)}.$$
\end{sled}

\begin{proof}
As before, the sets  $\mathfrak{F}(G)$ of $d$-tuples and  $\gamma(G)$ of proper orientations can be simultaneously partitioned into disjoint classes, and there exists a unique $d$-uniform $d$-factor $F$ for each class. Consequently if $\Phi(F) \leq \frac{\Delta(F)}{\mu(n,d)}$  for all 1-factorable $d$-uniform $d$-factors $F$, then the analogous inequality holds for the cardinalities of  $\mathfrak{F}(G)$  and $\gamma(G)$:  $$|\mathfrak{F}(G)| \leq \frac{|\gamma(G)|}{\mu(n,d)}.$$
\end{proof}

\begin{proof}[Proof of Theorem \ref{fact}]
Recall that  $\varphi^d(G) = |\mathfrak{F}(G)|.$ 
By Corollary~\ref{factdiag},  $|\mathfrak{F}(G)| \leq \frac{|\gamma(G)|}{\mu(n,d)}$. Also we know that $|\gamma(G)|$ is not greater than the permanent of the adjacency matrix $M(G)$. Therefore,
$$\varphi(G) \leq \left(\frac{\per M(G)}{\mu(n,d)} \right)^{1/d}.$$
\end{proof}

Let us begin the proof of Proposition~\ref{base} now. We show firstly that it is sufficient to consider only connected hypergraphs $F$.

\begin{lemma}
Suppose that for all connected 1-factorable $d$-uniform $d$-factors $F$ on $n$ vertices we have $\Phi(F) \leq \frac{\Delta(F)}{\mu(n,d)}$. Then this inequality holds for disconnected hypergraphs too.
\end{lemma}

\begin{proof}
Let $F_1, \ldots, F_k$ be all connected components of the hypergraph $F$: $F = F_1 \cup \ldots \cup F_k$.
Since $F$ is 1-factorable, the number of vertices in each component is a multiple of $d$. Denote these numbers by $n_1, \ldots, n_k$, $n_1 + \ldots + n_k = n.$ The hypergraphs $F_1, \ldots, F_k$ can be independently 1-factorized and oriented. Hence, $\Phi(F) = \Phi(F_1) \cdot \ldots \cdot \Phi(F_k)$ and $\Delta(F) = \Delta(F_1) \cdot \ldots \cdot \Delta(F_k)$. 

Suppose that for all $i \in \left\{1, \ldots, k \right\}$ it holds $\Phi(F_i) \leq \frac{\Delta(F_i)}{\mu(n_i,d)}$. Note that for the function $\mu(n,d)$ we have
$$\mu(n,d) = \mu(n_1,d) \cdot \ldots \cdot \mu(n_k,d).$$
Therefore,
 $$\Phi(F) \leq \frac{\Delta(F)}{\mu(n,d)}.$$
\end{proof}

Before proving Proposition~\ref{base}, we consider a simpler case when $F$ is a graph ($d=2$). In this case a connected 1-factorable 2-uniform 2-factor $F$ is an even cycle.  If $F$ has more than two vertices, then it has two 1-factorizations. Also, there are two possible proper orientations of edges. If $F$ has two vertices, then there exists a unique 1-factorization of $F$, and $F$ has only one proper orientation. Therefore if $F$ is a graph, then $\Phi(F) =\Delta(F)$.

To prove Proposition~\ref{base} we use the concept of bipartite representation of a hypergraph. For a hypergraph $G=(X,W)$, the \textit{bipartite representation} of $G$ is the bipartite graph $B(G)=(X,W;E)$  with the vertex set $X\cup W$, and $E$ is the edge set; the vertex $x \in X$ is adjacent to the vertex $w \in W$ in $B(G)$ if
and only if the vertex $x$ is incident to the edge $w$ in $G$. Note that the biadjacency matrix of $B(G)$ coincides with the incidence matrix of $G$. Also, if $G$ is a connected hypergraph, then its bipartite representation $B(G)$ is connected too.

Any bipartite graph can be considered as a bipartite representation of some hypergraph.  If $G$ is a $d$-uniform $d$-factor, then its bipartite representation $B(G)$ is a $d$-regular graph, and each row and each column of the adjacency matrix of $B(G)$ contains $d$ ones. 

Now we associate the numbers of 1-factorizations and proper orientations of a $d$-uniform $d$-factor $F$ with the numbers of some objects in its bipartite representation. For this purpose we need the following concepts.

Let $G=(V,E)$ be a graph. A \textit{proper edge coloring} with $k$ colors of the graph $G$ is an assignment of  ``colors'' to the edges of the graph so that no two adjacent edges have the same color. If $G$ is a $d$-regular bipartite graph, then each proper edge coloring of $G$ with $d$ colors  is equivalent to some 1-factorization of $G$.

Let  $B=(X,Y; E)$ be a $d$-regular bipartite graph with the parts $X$ and $Y$ such that $|X|=|Y|=n$, and let $d$ divide $n$. The \textit{proper decomposition} of the part $Y$ is a decomposition of $Y$ into disjoint 
subsets $Y_1, \ldots, Y_d$ such that the neighborhood of each $Y_i$ (i.e., the union of neighborhoods of $y$ over all $y \in Y_i$) is equal to $X$. In other words, each vertex $x \in X$ is adjacent to exactly one vertex from each $Y_i$. 

Recall that the $d$-uniform $d$-factor $F$ in Proposition~\ref{base} may contain multiple hyperedges that correspond  to the vertices $w \in W$ with identical neighborhoods in  $B(F)$. Suppose that there are $k$ different hyperedges in the hypergraph $F$, and let the $i$th hyperedge have the multiplicity  $l_i$, $i=1,\ldots, k.$ Let $R(F) = \prod\limits_{i=1}^k l_i!.$ 

Next we associate the number of proper orientations of $F$ with the number of proper edge colorings of $B(F)$:

\begin{lemma} \label{plem}
Suppose $F$ is a 1-factorable $d$-uniform $d$-factor and $B(F)$ is its bipartite representation. Denote by $P(B)$ the number of proper edge colorings of $B(F)$ with $d$ colors. Then
$$\Delta(F) = P(B) / R(F).$$
\end{lemma}

\begin{proof}
The correspondence between the proper edge colorings of $B(F)$ with $d$ colors and the proper orientations of $F$ can be given by the following rule: if $x \in X$ and $w \in W$ are connected by the edge with color $i$ in the graph $B(F)$, then the position of $x$ in the orientation of a hyperedge $w$ in the hypergraph $F$ equals $i$. Different proper edge colorings of $B(F)$ may correspond to the same  orientation of $F$ if and only if $F$ has multiple hyperedges, and such edge colorings can be changed from one to another by permutations of labels of vertices $w \in W$ with the same neighborhoods. So each proper orientation is counted  $R(F)$  times, and   $\Delta(F) = P(B)/R(F).$
\end{proof}

By Hall's marriage theorem,  every  bipartite $d$-regular graph has a proper edge coloring with $d$ colors. Therefore we have

\begin{sled}
Every 1-factorable $d$-uniform $d$-factor $F$ has a proper orientation.
\end{sled}

Now let us associate the number of 1-factorizations of $F$ with the number of proper decompositions of $B(F)$:

\begin{lemma}
Let $F$ be a 1-factorable $d$-uniform $d$-factor, and let $B(F)$ be its bipartite representation. Denote by $T(B)$ the number of proper decompositions of the part $W$ of $B(F)$. Then
$$\Phi(F) = T(B)/R(F).$$
\end{lemma}

\begin{proof}
The correspondence between the proper decompositions of the part $W$ of $B(F)$  and the 1-factorizations of $F$ can be given by the following rule: if  $w \in W$  belongs to the subset $W_i$ in a proper decomposition of $B(F)$, then the hyperedge $w$ belongs to the $i$th 1-factor in the 1-factorization of $F$. We get $\Phi(F) = T(B)/R(F)$ similarly to the proof of Lemma~\ref{plem}.
\end{proof}

Now to obtain Proposition~\ref{base} it is sufficient to prove the following lemma:

\begin{lemma} \label{dvud}
Let $B = (X,Y;E)$ be a  $d$-regular connected bipartite graph on $2n$ vertices, and let $d$ divide $n$. Then $T(B) \leq \frac{P(B)}{\mu(n,d)}$.
\end{lemma}

\begin{proof}
Firstly we prove the inequality for a simpler case $d\geq 4$. The case $d=3$ needs more precise estimations.

The first step is to find an upper bound on $T(B)$. For this purpose we estimate how many ways there exist to construct the subset $Y_1$ for a proper decomposition: a first vertex $x_1 \in X$ can be covered by any adjacent $y \in Y$. Hence the first vertex $y_1$ for $Y_1$ can be chosen by $d$ ways. Let  $x_2$ be a vertex that does not belong to the neighborhood of $y_1$. Then the number of vertices adjacent to $x_2$ is not greater than $d$, and there are at most $d$ ways to choose the second vertex $y_2$ for $Y_1$. Iterating this process, we see that there are at most $d^{n/d}$ ways to construct the set $Y_1$.

Delete all vertices $y_1, \ldots, y_{n/d}$ obtained at the previous step and all incident edges from the graph $B$. Then estimate analogously the number of ways to choose the set $Y_2$ for a proper decomposition in the remaining graph, and find that this number is at most $(d-1)^{n/d}$. By multiplying the estimations, we get that the number of proper decompositions of  the part $Y$ is not greater than  $d!^{n/d}$:
$$T(B) \leq d!^{n/d}.$$  

By Theorem~\ref{tchriv}, we have the number of proper edge colorings of $B$ with $d$ colors:
$$P(B) \geq \left(\frac{d!^2}{d^d}\right)^n. $$

Therefore,
$$\frac{T(B)}{P(B)} \leq \left( \frac{d!^{1/d}d^{d}}{d!^{2}} \right)^n = \frac{1}{\mu(n,d)}$$
for all $d$-regular connected bipartite graphs $B$ on $2n$ vertices.

Consider the case $d=3$. Now to estimate $T(B)$ we use the connectedness of the graph $B$. Let us construct subsets $Y_1, Y_2$, and $Y_3$ for a proper decomposition step by step.

Before the first step we choose a vertex $y^1_1$ for the subset $Y_1$ covering some vertex $x_1 \in X$. Note that this can be done in three ways.

Assume that after $k$ steps, $k \geq 0$,  we have $Y_i = \left\{y_i^1, \ldots, y_i^{m_i}\right\}$, $i \in \left\{1, 2,3\right\}$, where $m_1 + m_2 + m_3 \geq 2k+1,$ and there are no vertices $x \in X$ adjacent to exactly two vertices from $Y_1 \cup Y_2 \cup Y_3$.  
Since $B$ is connected, there exists a vertex $x' \in X$ such that $x'$ is covered by only one of vertices $Y_1 \cup Y_2 \cup Y_3$. Without loss of generality, let $x'$ be adjacent to some vertex from $Y_1$. Then there are 2 ways to find a vertex $y_2^{m_2+1}$ for the set $Y_2$ and a vertex $y_3^{m_3+1}$ for the set $Y_3$ covering the vertex $x'$.
Next if there exists, for example, a vertex $x''$ adjacent to some vertices $y_1^{j} \in Y_1$ and $y_2^{k} \in Y_2$ and non-adjacent to all vertices from $Y_3$, then we uniquely find the last vertex covering $x''$ and join it to the $Y_3$. Continue this process until every vertex $x \in X$ is adjacent to zero, one or three vertices from $Y_1 \cup Y_2 \cup Y_3$. 

Note that a final vertex in the construction of $Y_1$, $Y_2$, and $Y_3$ is defined uniquely. 
Since at each step a cardinality of $Y_1 \cup Y_2 \cup Y_3$ increases by at least 2, a total number of steps  does not exceed $\frac{n}{2} -1$. Therefore, we have the number of proper decompositions of the part $Y$: 
$$T(B) \leq 3 \cdot 2^{\frac{n}{2} -1}.$$

Let us estimate $P(B)$. Note that the permanent of a matrix with 2 ones in each row and column is not less than 2. Using this fact and Theorem~\ref{nper}, we obtain 
$$ P(B) \geq 2 \left(\frac{4}{3}\right)^n.$$
Then $$\frac{T(B)}{P(B)} \leq  \frac{3^{n+1}}{2^{3n/2 +2}} < \frac{3^n}{2^{3n/2}} = \frac{1}{\mu(n,3)}.$$
\end{proof}

Recall that Lemma~\ref{dvud} implies Proposition~\ref{base}, and the proof of Theorem~\ref{fact} is complete.

Next we prove the simple upper bound on the number of 1-factors of $k$-balanced $d$-partite  hypergraphs. But for the large number of vertices  this bound is weaker than Theorem~\ref{fact}. To state this result we need the concept of latin squares. 
 
A \textit{latin square of order $n$} is an $n \times n$ array of $n$ symbols, in which each symbol occurs exactly once in each row and each column. Denote by $L(n)$ the number of all latin squares of order $n$, and by $Q(n)$ the number of latin squares with the fixed filling of some column. Note that $L(n) =  Q(n)n!.$ 

Let us prove firstly the following lemma:
\begin{lemma} \label{latin}
Let $U(d)$ be a $d$-dimensional (0,1)-matrix of order $d$ such that $u_{\alpha} = 1$ if and only if all  $\alpha_1, \ldots, \alpha_d$  are different. Then the permanent of $U(d)$ is equal to the  number of latin squares with the fixed filling of one column:
$$\per U(d) = Q(d).$$  
\end{lemma}

\begin{proof}
Let the set of indices   $(\alpha^1,\ldots, \alpha^d)$ be a unity diagonal in the matrix $U$. Construct the $d \times d$ array $T$ such that  $t_{i,j} = \alpha^i_j.$ Since each $\alpha^i$ corresponds to a unit entry in the matrix $U(d)$, we have that each row of $T$ contains different elements. Since indices   $(\alpha^1,\ldots, \alpha^d)$ form a diagonal in $U$, we see that all elements in any column of $T$ are different.

Therefore the array  $T$ is a latin square of order $d$. Similarly, for each latin square of order $d$ we can construct the unity diagonal in $U(d)$. Note that permutations of rows of the square $T$ preserve the diagonal. Consequently the permanent of $U(d)$ is equal to the  number of latin squares with the fixed filling of one column.
\end{proof}

It is worth mentioning that the number of 1-factorizations of the complete bipartite graph $K_{d,d}$ is equal to the number of latin squares of order $d$ and equals  $ d! \cdot \per U(d).$

Using Lemma~\ref{latin}, let us estimate the number of 1-factors of a $k$-balanced $d$-partite  hypergraph: 

\begin{teorema}
Let $G$ be a simple $k$-balanced $d$-partite hypergraph. Then the number of 1-factors of $G$ satisfies
$$\varphi(G) \leq \left( \frac{\per M(G)}{Q(d)} \right)^{1/d}.$$
\end{teorema}

\begin{proof}
Let us enumerate the parts of $G$ from 1 to $d$. Note that the adjacency matrix $M(G)$ can be split into blocks $v_\beta$ of order $k$, $\beta=(\beta_1, \ldots, \beta_d)$, $\beta_i \in \left\{1, \ldots, d\right\}$,  such that an entry $m_\alpha $ belongs to the block $v_\beta$ if and only if the vertex with label $\alpha_i$ belongs to the $\beta_i$th part of the hypergraph $G$. Moreover, all unit entries of $M(G)$ belong to blocks $v_\beta$ for which all $\beta_1, \ldots, \beta_d$ are different. Then the $d$-dimensional (0,1)-matrix of order $d$ such that $u_\beta=1$ if and only if the block $v_\beta$ contains ones coincides with $U(d)$. 

Consider a 1-factor of the hypergraph $G$. There exists an orientation of the 1-factor hyperedges  such that corresponding entries of $M(G)$ form a partial unity diagonal of length $k$ in any block $v_\beta$. If   $f \in \mathfrak{F}(G)$ is a $d$-tuple of 1-factors in $G$  and $(\beta^1, \ldots, \beta^d)$ is a unity diagonal in matrix $U(d)$, then  we can orientate properly hyperedges of the $i$th 1-factor of $f$ so that corresponding entries of $M(G)$ form a partial unity diagonal in block $v_{\beta^i}$. Also the union of all entries corresponding to hyperedges of $f$ is a unity diagonal of $M(G)$.  
So the number of proper orientations of each $d$-tuple $f$ is not less than the permanent of $U(d)$.

By Lemma~\ref{latin}, the permanent of $U(d)$ equals $Q(d)$. Consequently there are at least $Q(d)$ unity diagonals in $M(G)$ for each $d$-tuple  $f  \in \mathfrak{F}(G)$. This implies that $|\mathfrak{F}(G)| \leq \per M(G) /Q(d).$ Since $\varphi^d(G) =|\mathfrak{F}(G)|$, we obtain
$$\varphi(G) \leq \left(\frac{\per M(G)}{Q(d)} \right)^{1/d}.$$
\end{proof}

At the end of this section we consider the following examples that illustrate the tightness of Theorem~\ref{fact}.

\textbf{Example 1.} Let $G$ be a $d$-uniform hypergraph on $d$ vertices with single hyperedge, $d \geq 4$. It is obvious that $\varphi(G) = 1.$ An entry $m_\alpha$ of the adjacency matrix $M(G)$ equals 1 if and only if all components of index $\alpha$ are different, so $M(G) = U(d)$. By Lemma~\ref{latin}, we have that $\per M(G) = Q(d) = L(d)/ d!.$ The classical lower bound on the number of latin squares of order $d$ is
$$L(d) \geq \frac{d!^{2d}}{d^{d^2}}.$$

But 
$$\mu(d,d) = \frac{d!^{2d}}{d^{d^2}d!},$$
and a significant improvement of the function $\mu(n,d)$ in this case implies a similar improvement of the lower bound on the number of latin squares.

\textbf{Example 2.} Consider the complete $d$-uniform hypergraph $G^d_n$ on $n$ vertices, $d \geq 4$. The number of 1-factors of $G^d_n$ is equal to the number of unordered partitions of the vertex set into $n/d$ disjoint groups of size $d$:
$$\varphi(G^d_n) = \frac{1}{(n/d)!} {{n}\choose{d, \ldots, d}} = \frac{n!}{d!^{n/d} (n/d)!}.$$

With the help of Stirling's approximation we obtain
$$\varphi(G^d_n) = e^{o(n)} \left( \frac{1}{(d-1)!} \frac{n^{d-1}}{e^{d-1}}\right)^{n/d}  \mbox{ as } n \rightarrow \infty. $$

The permanent of the adjacency matrix $M(G^d_n)$ is not greater than the permanent of the $d$-dimensional matrix of order $n$ all of whose entries are equal to 1. Since the permanent of such matrix equals $n!^{d-1}$, Theorem~\ref{fact} implies
$$\varphi(G^d_n) \leq \left(\frac{\per M(G^d_n)}{\mu(n,d)}\right)^{1/d} \leq \left(\frac{d^{dn} d!^{n/d} }{d!^{2n}} n!^{d-1}\right)^{1/d}.$$

Using Stirling's approximation again, we estimate the number of 1-factors of $G^d_n$ as follows:
$$\varphi(G^d_n) \leq e^{o(n)} \left( \frac{d!^{1/d} d^d}{d!^{2}} \frac{n^{d-1}}{e^{d-1}}\right)^{n/d}  \mbox{ as } n \rightarrow \infty.$$

\section{An upper bound on the number of 1-factorizations of complete hypergraphs}

Denote by $G^d_n$ the complete $d$-uniform hypergraph on $n$ vertices, i.e., the hyperedge set of $G^d_n$ is the set of all $d$-element subsets of the vertex set. Let $M(G^d_n)$ be the adjacency matrix of this graph, and let $\Phi(n,d)$ be  the number of its 1-factorizations. It is easy to check that  each 1-factorization of $G^d_n$ consists of $t={{n-1}\choose{d-1}}$  1-factors.

Recall that if a hypergraph has a 1-factor, then $n$ is a multiple of $d$. By Baranyai's theorem~\cite{bar}, this condition is sufficient for the existence of a 1-factorization of complete hypergraphs. 

First we prove the following trivial bound on the number of 1-factorizations:
\begin{utv}
The number of 1-factorizations of the hypergraph $G^d_n$ satisfies
$$\Phi(n,d) \leq \left( (1+o(1)) \frac{n^{d-1}}{(d-1)!}\right)^{\frac{n^d}{d!}} \mbox{ as } n \rightarrow \infty.$$
\end{utv}

\begin{proof}
Let $\varphi_1, \ldots, \varphi_t$ be a 1-factorization of the hypergraph $G^d_n$. Construct the coloring of unit entries of the adjacency matrix  $M(G^d_n)$ with $t$ colors by the next rule: an entry $m_\alpha$ has the color $i$ if and only if the hyperedge $\alpha = (\alpha_1, \ldots, \alpha_d)$ belongs to the 1-factor $\varphi_i$. Note that if the entry $m_\alpha$ has the color $i$, then for all $\beta$, such that $(\beta_1, \ldots, \beta_d)$ is a permutation of  $(\alpha_1, \ldots, \alpha_d)$, the entry $m_\beta$ has the color $i$ too. Therefore, to define a coloring  of  all unit entries of the adjacency matrix, it is sufficient to specify colors for at most  $n^d/d!$ entries.

The number of colorings of $n^d/d!$ entries of the matrix $M(G^d_n)$ with $t$ colors is equal to  $t^\frac{n^d}{d!}$. Since  $t = {{n-1}\choose {d-1}} = (1+o(1))\frac{n^{d-1}}{(d-1)!}$, we have 
$$\Phi(n,d) \leq \left( (1+o(1)) \frac{n^{d-1}}{(d-1)!}\right)^{\frac{n^d}{d!}} \mbox{ as } n \rightarrow \infty.$$
\end{proof}

Our reasoning for the estimation of the number of 1-factorizations of the complete hypergraphs will be similar to the proof for the complete graphs, but instead of the result of~\cite{alon} and Bregman's theorem we use Theorem~\ref{fact} and the following result of~\cite{my}.

Let $r(n)$ be an $n$-vector $(r_1(n), \ldots, r_n(n)).$
Denote by $\Lambda^d(n,r(n))$  the set of $d$-dimensional (0,1)-matrices of order $n$ such that the number of ones in their hyperplanes is not greater than $r_i(n).$

\begin{teorema} \label{asym}
Assume that for given integer $d \geq 2$ and for all integer $n$ there are $n$ integer numbers  $r_1(n), \ldots ,r_n(n)$ such that $\min\limits_{i=1 \ldots n} r_i(n)/ n^{d-2} \rightarrow \infty$ as $n \rightarrow \infty.$ Let $S(x) = \left\lceil x\right\rceil!^{1/\left\lceil x\right\rceil}.$ Then
$$ \max\limits_{A \in \Lambda^d(n,r(n))} \per A \leq n!^{d-2} e^{o(n)} \prod \limits_{i=1}^n S\left(\frac{r_i(n)}{n^{d-2}}\right) \mbox{ as } n \rightarrow \infty.$$
\end{teorema}

Now we are ready to prove the main theorem of this section:

\begin{teorema}
The number of 1-factorizations of the complete $d$-uniform hypergraph $G^d_n$ on $n$ vertices satisfies
$$\Phi(n,d) \leq \left((1+o(1))\frac{ n^{d-1}}{\mu(n,d)^{1/n} e^d} \right)^{\frac{n^d}{d!}} \mbox{ as }n \rightarrow \infty.$$
\end{teorema}

\begin{proof}
Let $\varphi_1, \ldots, \varphi_i$ be a set of $i$ disjoint 1-factors of the hypergraph $G^d_n$. Denote by $G_i$ the hypergraph $G^d_n \setminus \left\{ \varphi_1, \ldots, \varphi_i\right\}$, and let $M_i=M(G_i)$ be the adjacency matrix of the hypergraph $G_i$.

By Theorem~\ref{fact}, the number of 1-factors  of $G^d_n$ disjoint with  $\varphi_1, \ldots, \varphi_i$  is not greater than $\left(\frac{\per M_i}{\mu(n,d)} \right)^{1/d}.$ Then the number of 1-factorizations of  $G^d_n$ satisfies
$$\Phi(n,d) \leq \prod\limits_{i=0}^{t-1} \max\limits_{\varphi_1, \ldots, \varphi_i}  \left(\frac{\per M_i}{\mu(n,d)} \right)^{1/d},$$
where $M_0 = M(G^d_n)$, and the maximum is over all sets of disjoint 1-factors.

Let us find the number of ones in the hyperplanes of $M_i$. First we note that there are  $n(t-i)/d$ hyperedges in $G_i$. Each hyperedge of $G_i$ corresponds to $d!$ unit entries in the matrix $M_i$, and sets of corresponding entries for different hyperedges are disjoint.  Consequently there are $n(t-i)(d-1)!$ ones in the matrix  $M_i$. Since each vertex of the hypergraph $G_i$ has the same degree, we see that each hyperplane of $M_i$  contains the same number of ones, and this number is
$$R_i =  (t-i)(d-1)!.$$

It can be checked that for all $i$ from the interval $\Delta(l) = \left[l \frac{n^{d-2}}{(d-1)!}, (l+1) \frac{n^{d-2}}{(d-1)!}\right]$ the value of $R_i$ is not greater than $ (n-l) n^{d-2}.$   Notice that there are at most $\left\lceil \frac{n^{d-2}}{(d-1)!}\right\rceil$  values of $i$ belong to the interval $\Delta(l)$. 
Let the matrix $N(l)$ have the maximal permanent among all $M_i$ for $i \in \Delta(l)$. Change some zeros of the matrix $N(l)$ to ones so that each hyperplane of some direction contains   $ (n-l) n^{d-2}$  ones and denote the constructed matrix by $M(l)$. Obviously, $\per M_i \leq \per N(l) \leq \per M(l)$ for all  $i \in \Delta(l)$.

Therefore we can rewrite the upper bound on the number of 1-factorizations as follows:
$$\Phi(n,d) \leq   \mu(n,d)^{-t/d} \left(\prod\limits_{l=0}^{n-1} \per^{1/d} M(l) \right)^{\left\lceil \frac{n^{d-2}}{(d-1)!}\right\rceil}.$$

Split the product of the permanents into two parts: when $l$ belongs to the interval $\left[ 0, n-\sqrt{n} \right]$, and when $l$ is in $\left( n-\sqrt{n}, n-1\right]$. For the first part we use Theorem~\ref{asym}:
$$\prod\limits_{l=0}^{n-\sqrt{n}}  \per M(l)  \leq  \prod\limits_{l=0}^{n-\sqrt{n}} n!^{d-2} e^{o(n)} S^n\left(n-l\right) .$$

By the definition of the function $S$ and by Stirling's approximation, we get
\begin{gather*}
\prod\limits_{l=0}^{n-\sqrt{n}} e^{o(n)} S^n\left(n-l\right) = \prod\limits_{l=0}^{n-\sqrt{n}} e^{o(n)}  \left(n-l\right)! ^{\frac{n}{\left(n-l\right)}} \\  = \prod\limits_{l=0}^{n-\sqrt{n}} e^{o(n)} \left(n-l\right)^{n} e^{-n} \leq e^{-n^2+ o(n^2)} n!^{n} = e^{o(n^2)} \left(\frac{n}{e^2}\right)^{n^2}.
\end{gather*}

For the second part of the product we use Proposition~\ref{triv}:
$$\prod\limits_{l=n-\sqrt{n}}^n  \per M(l) \leq \left(n^{d-2}\sqrt{n}\right)^{n\sqrt{n}} =  e^{o(n^2)}. $$

Thus,
 $$\prod\limits_{l=0}^{n-1} \per^{1/d} M(l) \leq  \left(n!^{(d-2)n}e^{o(n^2)} \left(\frac{n}{e^2}\right)^{n^2}\right)^{1/d} .$$
Insert this bound into the inequality for 1-factorizations and obtain
\begin{gather*}
\Phi(n,d) \leq   \mu(n,d)^{-t/d} e^{o(n^d)} \left(\frac{n}{e^2}\right)^{\frac{n^d}{d!}} n!^{\frac{(d-2)n^{d-1}}{d!}} \\ = \mu(n,d)^{-t/d} e^{o(n^d)} \left(\frac{n^{d-1}}{e^{d}}\right)^ {\frac{n^d}{d!} }.
\end{gather*}

Recall that $t ={{n-1}\choose {d-1}} = \frac{n^{d-1}}{(d-1)!} + o(n^{d-1})$. Therefore,
$$\Phi(n,d) \leq\left((1+o(1))\frac{n^{d-1}}{\mu(n,d)^{1/n} e^d}\right)^{\frac{n^d}{d!}} \mbox{ as }n \rightarrow \infty.$$

\end{proof}

\begin{sled}
If $d = 3$, then the number of 1-factorizations of the complete $3$-uniform hypergraph $G^3_n$ on $n$ vertices  satisfies
$$\Phi(n,3) \leq \left((1+o(1))\frac{3 n^{2}}{2^{3/2} \cdot e^3} \right)^{\frac{n^3}{6}} \mbox{ as }n \rightarrow \infty.$$
If $d \geq 4$, then the number of 1-factorizations of $G^d_n$ satisfies
$$\Phi(n,d) \leq \left((1+o(1)) \left(\frac{ d}{ e} \right)^d \frac{n^{d-1}}{d!^{2-1/d}} \right)^{\frac{n^d}{d!}} \mbox{ as }n \rightarrow \infty.$$
\end{sled}

\section{Acknowledgments}
The author is grateful to V.N. Potapov for constant attention to this work and for useful discussions.
The work is supported by the Russian Science Foundation (grant 14--11--00555).


\begin{thebibliography}{99}

\bibitem{alon}
N. Alon, S. Friedland, The maximum number of perfect matchings in graphs with a given degree sequence, Electron. J. Combin. \textbf{15} (2008).

\bibitem{bar}
Zs. Baranyai, On the factorization of the complete uniform hypergraph, Infinite and Finite Sets, Proc. Coll. Keszthely, Colloquia Math. Soc. Janos Bolyai \textbf{10} (1973),  91--107.

\bibitem{bregman}
L.M. Bregman, Some properties of nonnegative matrices and their permanents, Soviet Math. Dokl. \textbf{14} (1973), 945--949  [Dokl. Akad. Nauk SSSR \textbf{211} (1973), 27--30].

\bibitem{cam}
P.J. Cameron, Parallelisms of complete designs, London Math. Soc. Lecture Note Ser., Vol. 23, Cambridge Univ. Press, Cambridge, 1976.

\bibitem{dow}
S.J. Dow, P.M. Gibson, An upper bound for the permanent of a 3-dimensional (0,1)-matrix, Proc. Amer. Math. Soc. \textbf{99} (1987), no.~1, 29--34.

\bibitem{egor}
G.P. Egorychev, Proof of the van der Waerden conjecture for permanents, Sib. Math. J. \textbf{22} (1981), 854--859.

\bibitem{falikman}
D.I. Falikman, A proof of the van der Waerden conjecture regarding the permanent of a doubly stochastic matrix, Math. Notes Acad. Sci. USSR \textbf{29} (1981), 475--479.

\bibitem{gibson} 
P.M. Gibson, Combinatorial matrix functions and 1-factors of graphs, SIAM J. Appl. Math. \textbf{19} (1970), 330--333.

\bibitem{gurvits}
L. Gurvits, Van der Waerden/Schrijver-Valiant like conjectures and stable (aka hyperbolic) homogeneous polynomials: one theorem for all. With a corrigendum, Electron. J. Combin. \textbf{15} (2008).

\bibitem{isr}
N. Linial, Z. Luria, An upper bound on the number of Steiner triple systems, Random Structures Algorithms \textbf{43}, Issue 4 (2013), 399--406.

\bibitem{lo}
A. Lo, K. Markstr\"om, Perfect matchings in 3-partite 3-uniform hypergraphs, J. Combin. Theory Ser. A \textbf{127} (2014), 22--57.

\bibitem{minc}
H. Minc, Permanents, Encyclopedia of Mathematics and Its Applications, Vol. 6, Addison-Wesley, Reading, Mass. 1978.

\bibitem{rodl}
V. R\"odl, A. Ruci\`nski, E. Szemer\`edi, Perfect matchings in large uniform hypergraphs with large minimum collective degree, J. Combin. Theory Ser. A \textbf{116} (2009), 613--636.

\bibitem{shriv}
A. Schrijver, Counting 1-factors in regular bipartite graphs, J. Combin. Theory Ser. B \textbf{72}, Issue 1 (1998), 122--135.

\bibitem{my}
A.A. Taranenko, Upper bounds on the permanent of multidimensional (0,1)-matrices, Sib. \'Elektron. Mat. Izv. \textbf{11} (2014), 958--965.

\bibitem{zhao}
A. Treglown, Yi Zhao, Exact minimum degree thresholds for perfect matchings in uniform hypergraphs, J. Combin. Theory Ser. A \textbf{119} (2012), 1500--1522.
\end{thebibliography}
\end{document}